\documentclass[11pt,a4paper]{article}
\usepackage[english]{babel}
\usepackage{amstext,amsthm,amsxtra,amsfonts,amssymb,amsmath}
\usepackage{mathtools}
\usepackage[mathcal]{eucal}
\usepackage{mathrsfs}

\DeclareFontFamily{OT1}{pzc}{} 
\DeclareFontShape{OT1}{pzc}{m}{it}{<-> s * [1.100] pzcmi7t}{}
\DeclareMathAlphabet{\mathpzc}{OT1}{pzc}{m}{it}

\usepackage{epsfig,color,psfrag}
\usepackage[dvips, bookmarksopen, colorlinks, linkcolor = blue,
urlcolor = red, citecolor = red, menucolor = blue]{hyperref}

\usepackage{algorithm}
\usepackage{algpseudocode}

\newcommand{\hh}{h} 
\newcommand{\h}[1]{ \hh \left( #1 \right) }
\newcommand{\JJ}{J}
\newcommand{\J}[2]{ \JJ _{#1} \left( #2 \right)  }
\newcommand{\Jd}[2]{ \JJ _{#1}^{\delta} \left( #2 \right)  }
\newcommand{\U}{\mathpzc{ U } }
\renewcommand{\H}{\mathpzc{ H } }
\renewcommand{\L}{ \ensuremath{ L^2 \left( \Omega \right) } }
\newcommand{\bv}{ \ensuremath{ BV \left( \Omega \right) } }
\newcommand{\pp}[1]{ \varPhi \left( #1 \right) }
\newcommand{\M}[2]{\mathcal{M}_{#1} \left( #2 \right) }

\newcommand{\lin}[2]{\mathscr{L} \left( #1 , #2 \right) }

\newcommand{\F}{F}
\newcommand{\FF}{\ensuremath{ F } }
\newcommand{\dF}{F'}
\newcommand{\Fd}{F^*}

\renewcommand{\a}{\alpha}
\newcommand{\w}{\omega}
\newcommand{\g}{\gamma}

\newcommand{\R}{\mathbb{R}}
\newcommand{\N}{\mathbb{N}}

\newcommand{\range}[1]{\mathscr{R}(#1)}
\newcommand{\dom}[1]{\mathscr{D}(#1)}
\newcommand{\domf}{\ensuremath{dom} \ \hh } 
\newcommand{\ordem}[1]{\mathcal{O} \left( #1 \right) }

\newcommand{\argmin}{\ensuremath{argmin}}

\newcommand{\tq}{\ensuremath{ \ \mid \ }}
\renewcommand{\S}[1]{ \mathcal{S} \left( #1 \right) }

\newcommand{\ub}{\ensuremath{\overline{u}}}
\newcommand{\ua}{u_{\alpha}}

\newcommand{\uad}{u_{\alpha}^{\delta}}

\newcommand{\yd}{y^{\delta}}

\newcommand{\sg}[1]{\partial \hh \left( #1 \right) }
\newcommand{\db}[3]{D_{\hh}^{#1} \left( #2 , #3 \right) }

\newcommand{\bola}[2]{\mathcal{B}_{#1} \left( #2 \right) }
\newcommand{\inner}[2]{ \left\langle #1 \mbox{ , } #2 \right\rangle}
\newcommand{\hinner}[2]{ \left\langle #1 \mbox{ , } #2 \right\rangle _{\H} }
\newcommand{\norma}[1]{\left\Vert #1 \right\Vert}
\newcommand{\norm}[1]{\left\Vert #1 \right\Vert ^{2}}

\newcommand{\unorm}[1]{\left\Vert #1 \right\Vert ^{2}_{\U}}
\newcommand{\lnorma}[1]{\left\Vert #1 \right\Vert_{\L}}
\newcommand{\lnorm}[1]{\left\Vert #1 \right\Vert ^{2}_{\L}}
\newcommand{\snorma}[1]{ \lvert #1 \rvert_{\bv}}

\newcommand{\dhnorma}[1]{\left\Vert #1 \right\Vert_{\H^*}}
\newcommand{\dhnorm}[1]{\left\Vert #1 \right\Vert ^{2}_{\H^*}}

\newtheorem{theorem}{Theorem}[section]
\newtheorem{lemma}[theorem]{Lemma}
\newtheorem{proposition}[theorem]{Proposition}
\newtheorem{corollary}[theorem]{Corollary}
\newtheorem{remark}[theorem]{Remark}
\newtheorem{definition}[theorem]{Definition}
\newtheorem{assumption}[theorem]{Assumption}

\begin{document}
\setcounter{footnote}{1}

\title{On Tikhonov functionals penalized by Bregman distances}

\author{
I.R.~Bleyer%
\thanks{Department of Mathematics, Federal University of St. Catarina,
P.O. Box 476, 88040-900 Florian\'opolis, Brazil
(\href{mailto:ismaelbleyer@gmail.com}{\tt ismaelbleyer@gmail.com},
 \href{mailto:acgleitao@gmail.com}{\tt acgleitao@gmail.com}).}
\ \ and \
A.~Leit\~ao$^\dag$}

\date{\small\today}

\maketitle

\begin{small}
\begin{abstract}
We investigate Tikhonov regularization methods for nonlinear ill-posed
problems in Banach spaces, where the penalty term is described by Bregman
distances. We prove convergence and stability results. Moreover, using
appropriate source conditions, we are able to derive rates of convergence
in terms of Bregman distances. 
We also analyze an iterated Tikhonov method for nonlinear problems, where
the penalization is given by an appropriate convex functional.
\end{abstract}

\end{small}

\section{Introduction} \label{sec:intro}

In this paper we study non-quadratic regularization methods for solving
ill-posed operator equations of the form
\begin{equation}\label{eq:princ}
\F(u)=y \ ,
\end{equation}
where $\F : \dom{\F} \subset \U \rightarrow \H$ is an operator between infinite
dimensional Banach spaces. Both linear and nonlinear problems are considered.

Tikhonov method is widely used to approximate solutions of inverse problems
modeled by operator equations in Hilbert spaces \cite{Tikhonov1963,
Groetsch1984}. In this article we investigate a Tikhonov methods, which
consist of the minimization of functionals of the type
\begin{equation}\label{eq:regula}
\Jd{\a}{u} = \dfrac{1}{2} \| \F(u) - \yd \| + \a \h{u} \ ,
\end{equation}
where $\a \in \R_{+}$ is called regularization parameter, $\h{\cdot}$ is a
proper convex functional, and the noisy data $\yd$ satisfy
\begin{equation}\label{eq:ruido} 
\| y - \yd \|2 < \delta \ .
\end{equation}

The method presented above represents a generalization of the classical
Tikhonov regularization. Therefore, the following questions arise:
\begin{itemize}
\item For $\a > 0$, does the solution \eqref{eq:regula} exist? Does the
solution depends continuously on the data $\yd$?
\item Is the method convergent? (i.e., if the data $y$ is exact and $\a \to 0$,
do the minimizers of \eqref{eq:regula} converge to a solution of
\eqref{eq:princ}?)
\item Is the method stable in the sense that: if $\a = \a(\delta)$ is chosen
appropriately, do the minimizers of \eqref{eq:regula} converge to a solution
of \eqref{eq:princ} as $\delta \to 0$?
\item What is the rate of convergence? How should the parameter
$\a = \a(\delta)$ be chosen in order to get optimal convergence rates?
\end{itemize}

The first point above is answered in \cite{Hofmann2007}. Throughout this
article we assume the following assumptions.

\begin{assumption}\label{assump:1}
\renewcommand{\labelenumi}{(A\arabic{enumi})}
\begin{enumerate} \
\item Given the Banach spaces $\U$ and $\H$ one associates the topologies
$\tau_{\U}$ and $\tau_{\H}$, respectively, which are weaker than the norm
topologies;
\item The topological duals of $\U$ and $\H$ are denoted by ${\U}^*$ and
$\H$, respectively;
\item The norm $\norma{\cdot}$ is sequentially lower semi-continuous with
respect to $\tau_{\H}$, i.e., for $u_k \rightarrow u$ with respect to the
$\tau_{\U}$ topology, $\h{u} \leq \liminf_{k} \h{u_k}$;
\item $\dom{F}$ has non-empty interior with respect to the norm topology
and is $\tau_{\U}$-closed. Moreover, $\dom{F} \cap \domf \neq \emptyset$;
\item $F : \dom{F} \subseteq \U \rightarrow \H$ is continuous from
$\left( \U , \tau_{\U} \right) $ to $\left( \H , \tau_{\H} \right)$;
\item The functional $\hh : [0,+\infty] \rightarrow \H$ is proper, convex,
bounded from below and $\tau_{\U}$ lower semi-continuous;
\item For every $M > 0 \mbox{ , } \a > 0$, the sets
$$
\M{\a}{M} = \left\lbrace  u \in \U \tq \Jd{\a}{u} \leq M \right\rbrace 
$$
are $\tau_{\U}$ compact, i.e. every sequence $(u_k)$ in $\M{\a}{M}$ has a
subsequence, which is convergent in $\U$ with respect to the $\tau_{\U}$
topology.
\end{enumerate}
\end{assumption}

The goal of this paper is to answer the last three questions posed above.
We obtain convergence rates and error estimates with respect to the generalized
\textit{Bregman distances}, originally introduced in \cite{Bregman1967}. Even
though this tool does not satisfy symmetry requirement nor the triangular
inequality, it is the main ingredient to this work.

This paper is organized as follow: In section \ref{sec:linear} we consider
the linear case and give quantitative estimates for the minimizers of
\eqref{eq:regula}, for exact and for noisy data. In section \ref{sec:nonlinear}
contains similar results as the section \ref{sec:linear} for nonlinear
problems. In section \ref{sec:iterative} we briefly discuss a iterative
method for the nonlinear case, the main results contains convergence analysis.

\section{Convergence analysis for linear problems} \label{sec:linear}

In this we consider only the linear case. Equation \eqref{eq:princ} will be
denoted by $\F u = y$, and the operator is defined from a Banach space to
a Hilbert space. The main results of this section were proposed originally
in \cite{Burger2004, resmerita2005}.

\subsection{Rates of convergence for source condition of type I}

Error estimates for the solution error can be obtained only under additional
smoothness assumption on the data, the so called \textit{source conditions}.
At a first moment we assume that $y \in \range{\FF}$ and let $\ub$ be an
$\hh$-minimizing solution by definition \ref{def_sol_min}. We assume that
there exist at least one element $\xi$ in $\sg{\ub}$ that belongs to the
range of adjoint of the operator $\FF$. Note that $\range{\Fd} \subseteq \U^*$
and $\sg{\ub} \subseteq \U^*$. Summarizing, we have
\begin{equation}\label{source:l1}
\xi \in \range{\Fd} \cap \sg{\ub} \neq \varnothing \, ,
\end{equation}
where $\ub$ is such that
\begin{equation}\label{source:l1eq}
\F \ub = y \ .
\end{equation}

We can rewrite the source condition \eqref{source:l1} as following: there
exist an element $\w \in \H$ such that $\xi = \Fd \w$. Note that under this
assumption we can define the dual pairing for $\psi, u \in \U^* \times \U$,
where $\psi \in \range{\Fd}$ as
\begin{equation*}
\inner{\psi}{u} = \inner{\Fd \nu}{u} := \hinner{\nu}{\F u} \ ,
\end{equation*}for some $\nu \in \H$.

\begin{theorem}[Stability]\label{theo:ls1}
Let \eqref{eq:ruido} hold and let $\ub$ be an $\hh$-minimizing solution of
\eqref{eq:princ} such that the source condition (\ref{source:l1}) and
\eqref{source:l1eq} are satisfied. Then, for each minimizer $\uad$ of
\eqref{eq:regula} the estimate
\begin{equation}\label{est_theo:ls1}
\db{\Fd \w}{\uad}{\ub} \leq \dfrac{1}{2\a} \left( \a \norma{\w}
   + \delta \right)^2
\end{equation}
holds for $\a > 0$. In particular, if $\a \sim \delta$, then
$\db{\Fd \w}{\uad}{\ub} = \ordem{\delta}$.
\end{theorem}
\begin{proof}
We note that $\norm{\F \ub - \yd} \leq \delta^2$, by \eqref{source:l1eq} and
\eqref{eq:ruido}. Since $\uad$ is a minimizer of the regularized problem
\eqref{eq:regula}, we have
$$
\dfrac{1}{2} \norm{ \F \uad - \yd } + \a \h{\uad} \ \leq \
\dfrac{\delta^2}{2} + \a \h{\ub} \, .
$$
Let $\db{\Fd \w}{\uad}{\ub}$ the Bregman distance between $\uad$ and $\ub$,
so the above inequality becomes
$$
\dfrac{1}{2} \norm{ \F \uad - \yd } + \a \left(  \db{\Fd \w}{\uad}{\ub} +
\inner{\Fd \w}{\uad-\ub}  \right) \ \leq \ \dfrac{\delta^2}{2} \, .
$$
Hence, using \eqref{eq:ruido} and Cauchy-Schwarz inequality we can derive the
estimate
$$
\dfrac{1}{2} \norm{ \F \uad - \yd } + \hinner{ \a \w}{\F \uad - \yd } +
\a \db{\Fd \w}{\uad}{\ub} \ \leq \ \dfrac{\delta^2}{2} +
\a \norma{\w} \delta \, .
$$
Using the the equality $\norm{a+b} = \norm{a} + 2\inner{a}{b} + \norm{b}$, it
is easy to see that
$$
\dfrac{1}{2} \norm{ \F \uad - \yd + \a \w} + \a  \db{\Fd \w}{\uad}{\ub}
\ \leq \ \dfrac{\a^2 }{2} \norm{\w } + \a \delta \norma{\w} +
\dfrac{\delta^2}{2} \ ,
$$
which yields \eqref{est_theo:ls1} for $\a > 0$.
\end{proof}

\begin{theorem}[Convergence]\label{theo:lc1}
If $\ub$ is an $\hh$-minimizing solution of \eqref{eq:princ} such that the
source condition (\ref{source:l1}) and \eqref{source:l1eq} are satisfied,
then for each minimizer $\ua$ of \eqref{eq:regula} with exact data, the
estimate
$$
\db{\Fd \w}{\ua}{\ub} \leq \dfrac{\a}{2}\norm{\w}
$$
holds true.
\end{theorem}
\begin{proof}
The proof is analogous to the proof of theorem \ref{theo:ls1}, taking
$\delta = 0$.
\end{proof}

\subsection{Rates of convergence for source condition of type II}

In this section we use a source condition, which is stronger than the one
used in previous subsection. This condition corresponds the existence of some
element $\xi \in \sg{\ub} \subset \U^*$ in the range of the operator $\Fd \F$,
i.e.
\begin{equation}\label{source:l2}
\xi \in \range{\Fd \F} \cap \sg{\ub} \neq \varnothing \, ,
\end{equation}
where $\ub$ is such that
\begin{equation}\label{source:l2eq}
\Fd \F \ub = \Fd y \, .
\end{equation}

Note that in \eqref{source:l2eq} we do not require $y \in \range{\F}$.
Moreover, the definition \ref{def_sol_min} is given in context of
least-squares solution.
The condition \eqref{source:l2} is equivalent to the existence of
$\w \in \U \backslash \left\lbrace 0 \right\rbrace$ such that
$\xi = \Fd \F \w $, where $\Fd$ is the adjoint operator of $\F$ and
$\Fd \F : \U \rightarrow \U^*$.

\begin{theorem}[Stability]\label{theo:ls2}
Let \eqref{eq:ruido} hold and let $\ub$ be an $\hh$-minimizing solution of
\eqref{eq:princ} such that the source condition \eqref{source:l2} as well
as \eqref{source:l2eq} are satisfied. Then the following inequalities hold
for any $\a > 0$:
\begin{equation}\label{temp2.11}
\db{\Fd \F\w}{\uad}{\ub} \leq \db{\Fd \F\w}{\ub - \a \w}{\ub}
+ \dfrac{\delta^2}{\a} + \dfrac{\delta}{\a} \sqrt{\delta^2
+ 2 \a \db{\Fd \F\w}{\ub - \a \w}{\ub} } ,
\end{equation}
\begin{equation}\label{temp2.11a}
\norma{ \F \uad -  \F \ub } \leq \a \norma{ \F \w } + \delta
+ \sqrt{\delta^2 + 2 \a \db{\Fd \F\w}{\ub - \a \w}{\ub} } \, .
\end{equation}
\end{theorem}
\begin{proof}
Since $\uad$ is a minimizer of \eqref{eq:regula}, it follows from algebraic
manipulation and from the definition of Bregman distance that
\begin{eqnarray}\label{temp1}
0 & \geq & \dfrac{1}{2}
  \left[ \norm{ \F \uad - \yd } - \norm{ \F u - \yd } \right]
  + \a \h{\uad} - \a \h{u} \nonumber \\
  &  =   & \dfrac{1}{2} \left[ \norm{ \F \uad }  - \norm{ \F u } \right]
  - \hinner{ \F \left( \uad - u \right)} {\yd} - \a \db{\Fd \F\w}{u}{\ub}
  \nonumber \\
  &      & + \ \a \hinner{ \F\w}{\F \left( \uad - u \right) }
  + \a \db{\Fd \F\w}{\uad}{\ub} \, .
\end{eqnarray}
Notice that
\begin{eqnarray*}
\norm{ \F \uad }  - \norm{ \F u }
& = & \norm{\F \left( \uad - \ub + \a \w \right) }
      - \norm{\F \left( u - \ub + \a \w \right) } \\
&   & + \ 2 \hinner{\F \uad - \F u}{\F \ub - \a \F \w} \, .
\end{eqnarray*}
Moreover, by \eqref{source:l2eq}, we have $\hinner{\F \left( \uad - u \right)}
{\yd - \F \ub } = \hinner{ \F \left( \uad - u \right) } { \yd - y }$.
Therefore, it follows from \eqref{temp1} that
\begin{eqnarray*}
&      & \dfrac{1}{2} \norm{\F \left( \uad - \ub + \a \w \right) }
         + \a \db{\Fd \F\w}{\uad}{\ub} \\
& \leq & \hinner{ \F \left( \uad - u \right) } { \yd - y }
         + \a \db{\Fd \F\w}{u}{\ub} + \dfrac{1}{2}
         \norm{\F \left( u - \ub + \a \w \right) }
\end{eqnarray*}
for every $u \in \U$, $\a \geq 0$ and $\delta \geq 0$.

Replacing $u$ by $\ub - \a \w$ in the last inequality, using \eqref{eq:ruido},
relations $\inner{a}{b} \leq | \inner{a}{b} | \leq \norma{a} \norma{b}$,
and defining $\gamma = \norma{\F \left( \uad - \ub + \a \w \right) }$ we obtain
$$
\dfrac{1}{2} \gamma^2 + \a \db{\Fd \F\w}{\uad}{\ub} \ \leq \
  \delta \gamma + \a \db{\Fd \F\w}{\ub - \a \w}{\ub} \, .
$$
We estimate separately each term on the left hand side by right hand side.
One of the estimates is an inequality in the form of a polynomial of the
second degree for $\gamma$, which gives us the inequality
$$
\gamma \leq \delta +
\sqrt{\delta^2 + 2 \a \db{\Fd \F\w}{\ub - \a \w}{\ub} } \, .
$$
This inequality together with the other estimate, gives us \eqref{temp2.11}.
Now, \eqref{temp2.11a} follows from the fact that
$\norma{\F \left( \uad - \ub \right) } \leq \gamma + \a \norma{ \F \w}$.
\end{proof}

\begin{theorem}[Convergence]\label{theo:lc2}
Let $\alpha \ge 0$ be given. If $\ub$ is a $\hh$-minimizing solution of
\eqref{eq:princ} satisfying the source condition \eqref{source:l2} as well
as \eqref{source:l2eq}, then the following inequalities hold true:
$$
\db{\Fd \F\w}{\ua}{\ub} \leq \db{\Fd \F\w}{\ub - \a \w}{\ub} \, ,
$$
$$
\norma{ \F \ua - \F \ub } \leq \a \norma{ F\w } +
  \sqrt{2 \a \db{\Fd \F\w}{\ub - \a \w}{\ub} } \, .
$$
\end{theorem}
\begin{proof}
The proof is analogous to the proof of theorem \ref{theo:ls2}, taking
$\delta = 0$. Notice that here $\a$ can be taken equal to zero.
\end{proof}

\begin{corollary}\label{cor:l}
Let the assumptions of the theorem \ref{theo:ls2} hold true. Further, assume
that $\hh$ is twice differentiable in a neighborhood $U$ of $\ub$ and there
there exists a number $M >0$ such that for any $v \in \U$ and $u \in U$ the
inequality
\begin{equation}\label{temp2.10}
\inner{\hh''(u)v}{v} \leq M \norm{v}
\end{equation}
hold true. Then, for the parameter choice $\a \sim \delta^{\frac{2}{3} }$ we have
$\db{\xi}{\uad}{\ub} = \ordem{ \delta^{\frac{4}{3} } }$. Moreover, for exact data
we have $\db{\xi}{\ua}{\ub} = \ordem{ \a^2 }$.
\end{corollary}
\begin{proof}
Using Taylor's expansion at the point $\ub$ we obtain
$$
\h{u} \ = \ \h{\ub} + \inner{\hh'(\ub)}{u - \ub} +
\dfrac{1}{2} \inner{\hh''(\mu)(u - \ub)}{u - \ub}
$$
for some $\mu \in [ u , \ub ]$.
Let $u = \ub - \a \w$ in the above equality. For sufficiently small $\a$, it
follows from assumption \eqref{temp2.10} and the definition of the Bregman
distance, with $\xi = \hh'(\ub)$, that
\begin{eqnarray*}
\db{\xi}{\ub - \a \w}{\ub}
&  =   & \dfrac{1}{2} \inner{\hh''(\mu)(- \a \w )}{- \a \w } \nonumber \\
& \leq & \a^2 \dfrac{M}{2} \unorm{\w} \, .
\end{eqnarray*}
Note that $\db{\xi}{\ub - \a \w}{\ub} =  \ordem{\a^2} $, so the desired rates
of convergence follow from theorems \ref{theo:ls2} and \ref{theo:lc2}.
\end{proof}

\section{Convergence analysis for nonlinear problems} \label{sec:nonlinear}

\renewcommand{\F}[1]{F \left( #1 \right) }
\renewcommand{\dF}[1]{F' \left( #1 \right) }
\renewcommand{\Fd}[1]{F^* \left( #1 \right) }
\newcommand{\cfnld}{\ensuremath{\dF{\ub}^* \dF{\ub} \w}}
\renewcommand{\cfnld}{\xi}

This section points out the convergence analysis for the nonlinear problems.
We need to assume a nonlinear condition. In contrast with other classical
conditions, the following analysis covers the case when both $\U$ and $\H$
are Banach spaces.

\begin{assumption}\label{assump:2}
Assume that an $\hh$-minimizing solution $\ub$ of \eqref{eq:princ} exist and
that the operator $\FF : \dom{\FF} \subseteq \U \rightarrow \H$ is G\^ateaux
differentiable. Moreover, assume that there exists $\rho > 0$ such that, for
every $u \in \dom{F} \cap \bola{\rho}{\ub}$
\begin{equation}\label{cone}
\norma{ \F{u} - \F{\ub} - \dF{\ub} \left( u - \ub \right) } \leq
c \db{\xi}{u}{\ub} \ , \ c > 0 
\end{equation}
and
$\xi \in \sg{\ub}$.
\end{assumption}
This assumption was proposed originally in \cite{Resmerita2006}.

\subsection{Rates of convergence for source condition of type I}

For nonlinear operators we cannot define a adjoint operator. Therefore the
assumptions are done with respect to the linearization of the operator $\FF$.
In comparison with the source condition \eqref{source:l1} introduced on
previous section, we assume that
\begin{equation}\label{source:nl1}
\xi \in \range{ \dF{\ub}^*  } \cap \sg{\ub} \neq \varnothing
\end{equation}
where $\ub$ solves.
\begin{equation}\label{source:nl1eq}
\F{\ub} = y \, .
\end{equation}

The derivative of operator $\FF$ is defined between the Banach space $\U$
and $\lin{\U}{\H}$, the space of the linear transformations from $\U$ to $\H$.
When we apply the derivative at $\ub \in \U$ we have a linear operator
$\dF{\ub} : \U \rightarrow \H$ and so we can define its adjoint,
$\dF{\ub}^* : \H^* \rightarrow \U^*$.

The source condition (\ref{source:nl1}) is stated as follows: There exists an
element $\w \in \H^*$ such that
\begin{equation}\label{source:nl1a}
\xi = \dF{\ub}^* \w \in \sg{\ub} \, .
\end{equation}

\begin{theorem}[Stability]\label{theo:nls1}
Let the assumptions \ref{assump:1}, \ref{assump:2} and relation
\eqref{eq:ruido} hold true. Moreover, assume that there exist
$\w \in {\H}^*$ such that \eqref{source:nl1a} is satisfied and
$c \dhnorma{\w} < 1 $. Then, the following estimates hold:
$$
\norma{ \F{\uad} - \F{\ub} } \leq 2 \a \dhnorma{\w} +
  2 \left( \a^2 \unorm{\w} + \delta^{2} \right)^{ \frac{1}{2} } \, ,
$$
$$
\db{\dF{\ub}^* \w}{\uad}{\ub} \leq \dfrac{2}{1-c\dhnorma{\w}}
 \left[ \dfrac{\delta^2}{2\a} + \a\unorm{\w} +
 \dhnorma{\w}\left( \a^2\unorm{\w} + \delta^2 \right)^{ \frac{1}{2} } \right] \, .
$$
In particular, if $\alpha \sim \delta$, then $ \norma{ \F{\uad} - \F{\ub} } =
\ordem{\delta}$ and $\db{\dF{\ub}^* \w}{\uad}{\ub} = \ordem{\delta}$.
\end{theorem}
\begin{proof}
Since $\uad$ is the minimizer of \eqref{eq:regula}, it follows from the
definition of the Bregman distance that
\begin{eqnarray*}
\dfrac{1}{2} \norm{ \F{\uad} - \yd } & \leq &
  \dfrac{1}{2} \delta ^2 - \a \left( \db{\dF{\ub}^* \w }{\uad}{\ub}
  + \inner{\dF{\ub}^* \w }{\uad - \ub} \right) .
\end{eqnarray*}
By using \eqref{eq:ruido} and \eqref{source:nl1eq} we obtain
\begin{eqnarray*}
\dfrac{1}{2} \norm{\F{\uad} - \F{\ub}} & \leq & \norm{ \F{\uad} -\yd }
+ \delta^2 \, .
\end{eqnarray*}

Now, using the last two inequalities above, the definition of Bregman distance,
the nonlinearity condition and the assumption $\left( c \dhnorma{\w} -
1 \right) < 0$, we obtain
\begin{eqnarray}
\dfrac{1}{4}\norm{\F{\uad}- \F{\ub}} & \leq &
  \dfrac{1}{2}\left( \norm{ \F{\uad} -\yd } + \delta^2 \right) \nonumber \\
  & \leq & \delta ^2 - \a \db{\dF{\ub}^* \w }{\uad}{\ub}
  + \a \inner{ \w }{- \dF{\ub} \left( \uad - \ub\right) } \nonumber \\
& \leq & \delta ^2 - \a \db{\dF{\ub}^* \w }{\uad}{\ub}
  + \a \dhnorma{\w }\norma{\F{\uad} - \F{\ub} } \nonumber \\
& & + \a \dhnorma{\w }\norma{\F{\uad} - \F{\ub} -
  \dF{\ub} \left( \uad - \ub \right) }  \nonumber \\
& = & \delta ^2 + \a \left( c \dhnorma{\w} -1 \right)
  \db{\dF{\ub}^* \w }{\uad}{\ub} \nonumber \\
& & + \ \a \dhnorma{\w}\norma{\F{\uad} - \F{\ub} }\label{temp2a} \\
& \leq &  \delta ^2 + \a \dhnorma{\w }\norma{\F{\uad} - \F{\ub} } \label{temp2}
\end{eqnarray}
From \eqref{temp2} we obtain an inequality in the form of a polynomial of
second degree) for the variable $\gamma = \norma{\F{\uad} - \F{\ub}}$. This
gives us the first estimate stated by the theorem. For the second estimate
we use \eqref{temp2a} and the previous estimate for $\gamma$.
\end{proof}

\begin{theorem}[Convergence]\label{theo:nlc1}
Let the assumptions \ref{assump:1} and \ref{assump:2} hold true. Moreover,
assume the existence of $\w \in {\H}^*$ such that \eqref{source:nl1a} is
satisfied and $ c \dhnorma{\w} < 1 $. Then, the following estimates hold:
$$
\norma{ \F{\ua} - \F{\ub} } \leq 4 \a \dhnorma{\w} \, ,
$$
$$
\db{\dF{\ub}^* \w}{\ua}{\ub} \leq \dfrac{4 \a \dhnorm{\w}}{1-c\dhnorma{\w}} \, .
$$
\end{theorem}
\begin{proof}
The proof is analogous to the proof of theorem \ref{theo:nls1}, taking
$\delta = 0$.
\end{proof}

\subsection{Rates of convergence for source condition of type II}

In this subsection we consider once again the source condition presented in
(\ref{source:l2}), i.e. we assume the existence of
$$
\xi \in \range{ \dF{\ub}^*\dF{\ub}  } \cap \sg{\ub} \neq \varnothing \, .
$$
The assumption above is equivalent the existence of an element $\w \in \U$
with
\begin{equation}\label{source:nl2}
\xi = \dF{\ub}^*\dF{\ub} \w \in \sg{\ub} \, .
\end{equation}

\begin{theorem}[Stability]\label{theo:nls2}
Let the assumptions \ref{assump:1}, \ref{assump:2} hold as well as estimate
\eqref{eq:ruido}. Moreover, let $\H$ be a Hilbert space and assume the
existence of an $\hh$-minimizing solution $\ub$ of \eqref{eq:princ} in
the interior of $\dom{F}$. Assume also the existence of $\w \in \U$ such
that \eqref{source:nl2} is satisfied and $c \norma{\dF{\ub}\w} < 1$.
Then, for $\a$ sufficiently small the following estimates hold:
$$
\| \F{\uad} - \F{\ub} \| \leq \a \norma{ \dF{\ub} \w } + g(\alpha,\delta) \, ,
$$
\begin{equation}\label{temp25}
\db{\cfnld}{\uad}{\ub} \leq \frac{\alpha s + (c s)^2/2
+ \delta g(\alpha,\delta) + cs \left( \delta + \alpha \norma{\dF{\ub}\w}
\right) } {\alpha \left( 1 - c \norma{\dF{\ub} \w} \right) } \, ,
\end{equation}
where $g(\a,\delta) = \delta + \sqrt{ \left( \delta + cs \right)^2 +
2 \a s \left( 1 + c \norma{\dF{\ub}\w} \right) }$ and $s = \db{\cfnld}
{\ub - \a \w}{\ub}$.
\end{theorem}
\begin{proof}
Since $\uad$ is the minimizer of \eqref{eq:regula}, it follows that
\begin{eqnarray}\label{temp8}
0 & \geq & \dfrac{1}{2} \norm{ \F{\uad} - \yd } - \dfrac{1}{2}
  \norm{ \F{u} - \yd } + \a \left(  \h{\uad} - \h{u} \right) \nonumber \\
  &   =  & \dfrac{1}{2} \norm{\F{\uad}} - \dfrac{1}{2} \norm{ \F{u} } +
  \hinner{\F{u} - \F{\uad}}{\yd} \nonumber \\
  &      & + \ \a \left(  \h{\uad} - \h{u} \right) \nonumber \\
  &   =  & \pp{\uad} - \pp{u} \, .
\end{eqnarray}
where $\pp{u} = \dfrac{1}{2} \norm{ \F{u} - q } + \a \db{\cfnld}{u}{\ub} -
\hinner{\F{u}}{\yd - q} + \a \inner{\cfnld}{u} $,
$q = \F{\ub} - \a \dF{\ub}\w$ and $\xi$ is given by source condition
\eqref{source:nl2}.

From \eqref{temp8} we have $\pp{\uad} \leq \pp{u}$. By the definition of
$\pp{\cdot}$, taking $u = \ub -\a \w$ and setting $v = \F{\uad} - \F{\ub}
+ \a \dF{\ub}\w $ we obtain
\begin{eqnarray}\label{temp12}
\dfrac{1}{2} \norm{ v } + \a \db{\cfnld}{\uad}{\ub} & \leq & \a s
+ T_1 + T_2 + T_3 \, ,
\end{eqnarray}
where $s$ is given in the theorem, and
$$
T_1 = \dfrac{1}{2} \norm{ \F{\ub - \a \w } - \F{\ub} + \a \dF{\ub}\w } \, ,
$$
$$
T_2 = \left| \hinner{\F{\uad} -  \F{\ub - \a \w }}{\yd -  y } \right| \, ,
$$
$$
T_3 = \a \hinner{\dF{\ub}\w}{\F{\uad} - \F{\ub - \a \w }
- \dF{\ub} \left( \uad - \left( \ub - \a \w \right) \right)} \, .
$$
The next step is to estimate each one of the constants $T_j$ above. We use the
nonlinear condition \eqref{cone}, Cauchy-Schwarz, and some algebraic
manipulation to obtain $T_1 \leq \frac{c^2 s^2}{2}$,
\begin{eqnarray*}
T_2 & \leq & \left| \hinner{ v }{\yd -  y } \right| +
\left| \hinner{ \F{\ub - \a \w } - \F{\ub} + \a \dF{\ub}\w -  }{\yd -  y }
\right| \nonumber \\
& \leq & \norma{v} \norma{\yd -  y} + c \db{\cfnld}{\ub - \a \w}{\ub}
\norma{\yd -  y} \nonumber \\
& \leq & \delta \norma{v} + \delta c s \, ,
\end{eqnarray*}
and
\begin{eqnarray*}
T_3 & = &\a \hinner{\dF{\ub}\w}{\F{\uad} - \F{\ub} - \dF{\ub}
\left( \uad - \ub \right) }  \nonumber \\
& & + \a \hinner{\dF{\ub}\w}{ - \left( \F{\ub - \a \w } - \F{\ub} +
\a \dF{\ub} \w \right)  } \nonumber \\
& \leq & \a \norma{\dF{\ub}\w} \norma{\F{\uad} - \F{\ub} - \dF{\ub}
\left( \uad - \ub \right) }  \nonumber \\
& & + \a \norma{\dF{\ub}\w} \norma{\F{\ub - \a \w } - \F{\ub} +
\a \dF{\ub} \w } \nonumber \\
& \leq & \a \norma{\dF{\ub}\w} c \db{\cfnld}{\uad}{\ub} +
\a \norma{\dF{\ub}\w} c \db{\cfnld}{\ub - \a \w}{\ub} \nonumber \\
& = &\a c \norma{\dF{\ub}\w}  \db{\cfnld}{\uad}{\ub} + \a c s
\norma{\dF{\ub}\w} \, .
\end{eqnarray*}
Using these estimates in \eqref{temp12}, we obtain
\begin{eqnarray*}
\norm{ v } + 2 \a \db{\cfnld}{\uad}{\ub}  \left[ 1 - c \norma{\dF{\ub}\w}
\right]  & \leq & 2 \delta \norma{v} + 2 \a s + (c s) ^2 \nonumber \\
& & +  2 \delta c s + 2 \a c s \norma{\dF{\ub}\w} \, .
\end{eqnarray*}
Analogously as in the proof of theorem \ref{theo:ls2}, each term on the left
hand side of the last inequality is estimated separately by the right hand
side. This allows the derivation of an inequality described by a polynomial
of second degree. From this inequality, the theorem follows.
\end{proof}

\begin{theorem}[Convergence]\label{theo:nlc2}
Let assumptions \ref{assump:1}, \ref{assump:2} hold and assume $\H$ to be
a Hilbert space. Moreover, assume the existence of an $\hh$-minimizing
solution $\ub$ of \eqref{eq:princ} in the interior of $\dom{F}$, and also
the existence of $\w \in \U$ such that \eqref{source:nl2} is satisfied,
and $c \norma{\dF{\ub}\w} < 1$.
Then, for $\a$ sufficiently small the following estimates hold:
$$
\norma{\F{\ua} - \F{\ub} } \leq \a \norma{ \dF{\ub} \w } +
\sqrt{\left( cs \right)^2 + 2 \a s \left( 1 + c \norma{\dF{\ub}\w} \right)}\, ,
$$
\begin{equation}\label{temp28}
\db{\cfnld}{\ua}{\ub} \leq \frac{\alpha s + (c s)^2/2 +
\a c s \norma{\dF{\ub}\w}_{\H} } {\alpha \left( 1 - c \norma{\dF{\ub} \w}_{\H}
\right) } \, ,
\end{equation}
where $s = \db{\cfnld}{\ub - \a \w}{\ub}$.
\end{theorem}
\begin{proof}
The proof is analogous to the proof of theorem \ref{theo:nls2}, taking
$\delta = 0$.
\end{proof}

\begin{corollary}\label{cor:nl}
Let assumptions of the theorem \ref{theo:nls2} hold true. Moreover, assume
that $\hh$ is twice differentiable in a neighborhood $U$ of $\ub$, and that
there exist a number $M >0$ such that for all $u \in U$ and for all $v \in \U$,
the inequality $ \inner{\hh''(u)v}{v} \leq M \norm{v}$ hold.
Then, for the choice of parameter
$\a \sim \delta^{\frac{2}{3} }$ we have $\db{\xi}{\uad}{\ub} =
\ordem{ \delta^{\frac{4}{3} } }$, while for exact data we obtain
$\db{\xi}{\uad}{\ub} = \ordem{\a^2}$ .
\end{corollary}
\begin{proof}
The proof is similar to the proof of corollary \ref{cor:l} and we use
theorems \ref{theo:nls2} and \ref{theo:nlc2}.
\end{proof}

\section{An iterated Tikhonov method for nonlinear problems}
\label{sec:iterative}

\renewcommand{\F}[1]{F \left( #1 \right) }
\renewcommand{\dF}[1]{F' \left( #1 \right) }

On this section we investigate an iterative method based on Bregman distances
for nonlinear problems. We consider the operator $\FF : \U \rightarrow \H$
defined between a Banach space and a Hilbert space, Fr\'echet differentiable
with closed and convex domain $\dom{\FF}$. The operator equation
\eqref{eq:princ} is ill-posed in the sense of Hadamard, the solution does
not need to be unique, so we define
$$
\S{y} = \left\lbrace  u \in \dom{\FF} \ \mid \ \F{u} = y \right\rbrace \, .
$$

The method was originally proposed by Osher in \cite{Osher2005}, who
generalized the ideas of the method ROF \cite{Rudin1992}. One important
reference is \cite{Bachmayr2007}.

The analyzed method generalizes the iterated Tikhonov method, it is given by
\begin{equation}\label{alg1}
u_{k+1} \in \argmin \left\lbrace \dfrac{1}{2} \norm{ \F{u} - \yd } +
\a_k \db{\xi_k}{u}{u_k} \right\rbrace \, ,
\end{equation}
where the subgradient required is updated by the rule
\begin{equation}\label{alg2} 
\xi_{k+1} = \xi_k - \dfrac{1}{\a_k} \dF{u_{k+1}}^*
\left( \F{u_{k+1}} - \yd \right) \, .
\end{equation}

\begin{algorithm}
\mbox{\vskip-1.05cm \qquad\qquad\qquad Generalized Tikhonov with Bregman distance} \label{alg:iterative}
\begin{algorithmic}[1]
\Require $u_0 \in \dom{\FF} \cap \domf$, $\xi_0 \in \sg{u_0}$
\State $ k = 0 $
\State $\a_k > 0$
\Repeat
\State $ u_{k+1} \in \argmin \left\lbrace \dfrac{1}{2} \norm{ \F{u} - \yd }
       + \a_k \db{\xi_k}{u}{u_k} \right\rbrace $ 
\State $ \xi_{k+1} = \xi_k - \dfrac{1}{\a_k} \dF{u_{k+1}}^*
       \left( \F{u_{k+1}} - \yd \right) $
\State $k = k+1$
\State $\a_k > 0$
\Until convergence
\end{algorithmic}
\end{algorithm}

\begin{remark}\label{remark:update}
It is easy to see that the definition \eqref{alg2} is equivalent to
\begin{equation}\label{alg2mod}
\xi_{k+1} = \xi_0 - \sum_{j=0}^{k} \dfrac{1}{\a_j} \dF{u_{j+1}}^*
\left( \F{u_{j+1}} - \yd \right) \, .
\end{equation}
\end{remark}

We obtain monotonicity of residuals directly from the above definitions.

\begin{lemma}\label{lemma:nli1}
The iterates defined by algorithm \ref{alg:iterative} satisfy the estimate
$$
\norma{ \yd - \F{ u_{k+1} } } \leq \norma{ \yd - \F{ u_{k} } } \, .
$$
\end{lemma}
\begin{proof}
Defining $\Jd{\a}{u} = \dfrac{1}{2} \norm{ \F{u} - \yd } +
\a_k \db{\xi_k}{u}{u_k}$, the lemma follows the fact that $u_{k+1}$ is
a minimizer of \eqref{alg1}, i.e., $\Jd{\a}{u_{k+1}} \leq \Jd{\a}{u_k}$.
\end{proof}

Under a nonlinearity condition on $\FF$ we prove a monotonicity result for
the Bregman distance, i.e., $\db{\xi_{k+1}}{\ub}{u_{k+1}} \leq
\db{ \xi_k }{\ub}{u_k}$.

\begin{lemma}\label{lemma:nli2}
Let $\yd \in \H$ be given the data. If for some $u_k$ and $\xi_k$, the
iterate $u_{k+1}$ in \eqref{alg1} satisfies
$$
\norma{\yd - \F{u_{k+1}} - \dF{u_{k+1}} \left( \ub - u_{k+1} \right) } \leq
c \norma{\yd - \F{u_{k+1}}} \, ,
$$
for some $0 < c < 1$, then
\begin{equation}\label{eq:3ptos}
\db{\xi_{k+1}}{\ub}{u_{k+1}} - \db{ \xi_k }{\ub}{u_k} +
\db{\xi_{k}}{ u_{k+1} }{ u_k } \leq - \dfrac{1-c}{\a_k}
\norm{\yd - \F{ u_{k+1} } } \, .
\end{equation}
\end{lemma}
\begin{proof}
This result follows from the equality (see \cite{Bachmayr2007} for details)
\begin{equation*}
\db{\xi{k+1}}{\ub}{u_{k+1}} - \db{\xi_k}{\ub}{u_k} + \db{\xi_k}{u_{k+1}}{u_k} =
\inner{\xi_{k+1} - \xi_k}{u_{k+1} - \ub} \, .
\end{equation*}
Using \eqref{alg2} on the right hand side, summing $\pm \F{u_{k+1}} - \yd$
on the second term (inside the inner product), using Cauchy-Schwarz and
the lemma assumptions, we conclude that estimate \eqref{eq:3ptos} holds.
\end{proof}

The subsequent results are obtained assuming that the nonlinear operator $\FF$
is such that $\dom{\FF} \subseteq \L$ and $\Omega \subset \R^n$ is a bounded
Lipschitz domain, and assuming that the regularization convex functional is
given by
\begin{equation}\label{funcional}
\h{u} = \dfrac{1}{2} \lnorm{u} + \snorma{u} \, .
\end{equation}

\begin{lemma}\label{lemma:nli3}
If $\h{\cdot}$ is a convex functional defined by \eqref{funcional}, then
$$
\dfrac{1}{2} \lnorm{v-u} \leq \db{\xi}{v}{u}
$$
for every $u,v \in \dom{\FF}$ and $\xi \in \sg{u}$.
\end{lemma}
\begin{proof}
This proof is straightforward, once we establish some auxiliary properties
concerning calculus of subgradients. For a complete proof we refer the reader
to \cite{Bleyer2008}.
\end{proof}

\begin{assumption}\label{assum:3}
Let $\FF : \dom{\FF} \subset \L \rightarrow \H$ be a weakly sequentially
closed nonlinear operator, $\dF{\cdot}$ be locally bounded. Moreover, suppose
that the nonlinearity condition
\begin{equation}\label{eq:nlinear}
\norma{\F{v} - \F{u} - \dF{u}\left( v - u \right) } \leq \eta \lnorma{u - v}
\norma{\F{u} - \F{v}}
\end{equation}
is satisfied for every $u$, $v \in \bola{\rho}{\ub} \cap \dom{\FF}$, where
$\eta, \rho > 0 $ and $\bola{\rho}{\ub}$ denotes the open ball around $\ub$
of radius $\rho$ in $\L$ and $\ub \in \S{y} \cap \domf$.
\end{assumption}

\begin{remark}\label{remark:assumption}
We can rewrite the left side of the inequality given in (\ref{eq:nlinear}) as
\begin{equation*}
\norma{ \dF{u}\left( v - u \right) } \leq \left( 1 + \eta \lnorma{u - v}
\right)  \norma{\F{u} - \F{v} } \, .
\end{equation*}
\end{remark}

The next result gives the mean result about the sequence of iterates from
algorithm \ref{alg:iterative} is well-defined.

\begin{proposition}
Let assumption \ref{assum:3} hold, $k \in \N$ and $u_k, \xi_k$ be a
pair of iterates according to algorithm \ref{alg:iterative}. Then, there
exists a minimizer $u_{k+1}$ for \eqref{alg1} and $\xi_{k+1}$ given by
\eqref{alg2} satisfies $\xi_{k+1} \in \sg{u_{k+1}}$.
\end{proposition}
\begin{proof}
If there exist a $u$ such that $\Jd{\a_k}{u}$ is finite, then there is a
sequence $(u_j) \in \dom{\FF} \cap \bv$ such that $\lim_j \Jd{\a_k}{u_j}
\rightarrow \beta$, where $\beta = \inf \left\lbrace \Jd{\a_k}{u} \tq
u \in \dom{\FF} \right\rbrace $. In particular, $\db{\xi_k}{u_j}{u_k}
\leq \dfrac{M}{\a_k}$. By definition of the Bregman distance, together with
\eqref{funcional} and observing that $\dfrac{1}{2} \lnorm{u_j} -
\inner{\xi_k}{u_j} = \dfrac{1}{2} \lnorm{u_j - \xi_k} - \dfrac{1}{2}
\lnorm{\xi_k} $, we obtain $\snorma{u_j} \leq \tilde{M}_k $, where
$\tilde{M}_k \geq 0$ depends on the current iterates. Thus, the existence
of a minimizer follows from compactness arguments.

It remains to prove that $\xi_{k+1} \in \sg{u_{k+1}}$. This result follows from
the inequality $\phi_2 (v) \geq \phi_2(u_{k+1}) +
\inner{-\phi'_1(u_{k+1})}{v-u_{k+1}}$, where $\phi_1(u) = \dfrac{1}{2}
\lnorm{\yd - \F{u} }$ and $\phi_2(u) = \a_k \db{\xi_k}{u}{u_k}$ (see
\cite{Bachmayr2007, Bleyer2008} for details).
\end{proof}

\subsection{Main results}

The main results of this section give sufficient conditions to guarantee
existence of a convergence subsequence in algorithm \ref{alg:iterative},
(for both exact and noisy data). In particular, for noisy data, we introduce
a stopping rule based on the discrepancy principle. For a complete proof we
refer the reader to \cite{Bleyer2008}.

\begin{theorem}[Convergence] \label{theo:nlic}
Let the assumption \ref{assum:3} hold, $\g < \min \left\lbrace \frac{1}{\eta},
\frac{\rho}{2} \right\rbrace$ for $\eta$, $\rho$ as in \eqref{eq:nlinear},
$0 < \a_k < \bar{\a}$, $\h{\ub} < \infty $. Moreover, assume that the starting
values $u_0$, $\xi_0 \in \L$ satisfy $\db{\xi_0}{\ub}{u_0} < \frac{\g^2}{8}$
for some $\ub \in \S{y}$. Then, for exact data, the sequence $( u_k )$ has a
subsequence converging to some $u \in \S{y}$ in the weak-$*$ topology of
$\bv$. Moreover, if $\S{y} \cap \overline{ \bola{\rho}{\ub} } = \left\lbrace
\ub\right\rbrace$, then $u_k \xrightharpoonup{*} \ub$ in $\bv$.
\end{theorem}
\begin{proof}
\textbf{Step 1:} First we rewrite the assumption in the form
$2 \sqrt{2 \db{\xi_0}{\ub}{u_0} } < \g $. Assuming that the same condition
holds for a pair of iterates $u_k$, $\xi_k$ we proof by induction that it
also holds for the index $k+1$.

Let $u_{k+1}$ be the minimizer of $\J{\a_k}{\cdot}$, so $\J{\a_k}{u_{k+1}}
\leq \J{\a_k}{\ub}$. Thus we rewrite the inequality, then apply lemma
\ref{lemma:nli3} twice, and conclude that $\lnorma{u_{k+1} - \ub} < \g$.
Hence, assumption \ref{assum:3} is satisfied and the lemma \ref{lemma:nli2}
hold for all iterates.

\noindent
\textbf{Step 2:} In this step we proof that $\sum_{i=0}^{\infty} \dfrac{1}{\a_i}
\norm{y - \F{u_{i+1}} } < \infty$. \\
As in the previous step, by the lemma \ref{lemma:nli2} the inequality
\eqref{eq:3ptos} holds for every $k$. So we can sum up until $k$, for
some $k \in \N$. After that, we cancel the equal terms, apply the
assumption on starting values on the right hand side, and obtain
$$
\db{\xi_{k}}{\ub}{u_{k}} + \sum_{i=0}^{k-1} \db{\xi_{i}}{ u_{i+1} }{ u_i } +
\sum_{i=0}^{k-1} \dfrac{1 - \eta \g}{\a_i}\norm{y - \F{ u_{i+1} } } \leq
\dfrac{\g^2}{8} \, .
$$
Since all terms on the left hand side are positive, step 2 follows from
the third term taking the limit as $k$ tends to infinity. Note that this
series is convergent, by the convergence criterion for series follows
$\F{u_{k}} \rightarrow y$.

\noindent
\textbf{Step 3:} We show the uniform limitation of the sequence
$\left( \h{u_k} \right)$. Applying the Bregman distance (it is always
grater than zero) we have $\h{u_k} \leq \h{\ub} - \inner{\xi_k}{\ub - u_k}$.
Thus, by remark \ref{remark:update}, $\lnorma{u_{k+1} - \ub} < \g$ and the
Cauchy-Schwarz inequality, we obtain
$$
\h{u_k} \leq \h{\ub} + \g \lnorma{ \xi_0 } +
\sum_{i=0}^{k-1} \dfrac{1}{\a_i} \norma{ \F{u_{i+1}} - y }
\norma{ \dF{u_{i+1}} \left( \ub - u_k\right) } \, .
$$
In order to estimate the term inside the sum, note that for
$0 \leq i \leq k-1$, the estimate $ \norma{ \dF{u_{i+1}}
\left( \ub - u_k\right) } \leq \norma{ \dF{u_{i+1}}
\left( \ub - u_{i+1} \right) } + \norma{ \dF{u_{i+1}}
\left( u_k - u_{i+1} \right) } $ holds. Now, using remark
\ref{remark:assumption} twice, we find the bound
$\left( 3 + 5 \eta \g \right) \norma{ \F{u_{i+1}} - y}$ for the previous
estimate. Substituting this estimate in the sum above and using step 2,
the desired boundedness of the sequence $\left( \h{u_k} \right)$ follows.

\textbf{Step 4:} We know that $\lvert \h{u_k} \rvert = \h{u_k} \leq N $,
for some $N > 0$ (see \eqref{funcional}). The remaining assertions of the
theorem follow from standard compactness results (Banach-Alaoglu theorem).
We use the closed graph theorem to ensure that the limit of the obtained
sequence belongs to $\S{y}$.
\end{proof}

In the case of noisy data we use a generalized discrepancy principle as
stopping rule. The stopping index is defined as the smallest integer $k^*$
satisfying
\begin{equation}\label{eq:stop}
\| \F{u_{k^*}} - \yd \|  \leq  \tau \delta
\end{equation}
where $\tau > 1$ still has to be chosen.

\begin{theorem}[Stability] \label{theo:nlis}
Let assumption \ref{assum:3} hold, $\g < \min \left\lbrace \frac{1}{\eta} ,
\frac{\rho}{2} \right\rbrace $ for $\eta$ and $\rho$ as in \eqref{eq:nlinear},
$0 < \underline{\a} \leq \a_k \leq \overline{\a}$, $\h{\ub} < \infty $ and
the starting values $u_0$, $\xi_0 \in \L$ satisfy $\db{\xi_0}{\ub}{u_0} <
\frac{\g^2}{8}$ for an $\ub \in \S{y}$. Moreover, let $\delta_m > 0$ be a
sequence such that $\delta_m \rightarrow 0$, and let the corresponding
stopping indices $k_m^*$ be chosen according to \eqref{eq:stop} with
$\tau > (1 + \eta \g) / (1 - \eta \g)$.
Then for every $\delta_m$ the stopping index is finite and the sequence
$\left(  u_{k_m^*} \right) $ has a subsequence converging to an $u \in \S{y}$
in the weak-$*$ topology of $\bv$. Moreover, if $\S{y} \cap
\overline{ \bola{\rho}{\ub} } = \left\lbrace \ub \right\rbrace $, then
$u_{k_m^*} \xrightharpoonup{*} \ub$ in $\bv$.
\end{theorem}
\begin{proof}
\textbf{Step 1:} This step is analogous to step 1 in the proof of theorem
\ref{theo:nlic}.
For each $k$ such that $k < k^* - 1$, we have $\norma{ \F{u_k} - \yd } >
\tau \delta$. By induction one can prove that $\lnorma{ u_{k+1} - \ub } < \g$,
and that the nonlinear condition \eqref{eq:3ptos} holds. Therefore, lemma
\ref{lemma:nli2} holds for $c = \frac{1}{\tau} \left( 1 + \eta \g \right) +
\eta \g$. \\
\textbf{Step 2:} We show that the stopping index $k^*$ is finite. Analogous
to step 2 in the proof of theorem \ref{theo:nlic}, we sum up the first
$k^* - 1$ terms of \eqref{eq:3ptos}, obtaining
\begin{equation}\label{temp4}
\sum_{i=0}^{k^*-2} \dfrac{1}{\a_i}\norm{\yd - \F{ u_{i+1} } } <
\dfrac{\g^2}{8 \left( 1 - c \right)} \, .
\end{equation}
Since for every $k < k^* - 1$ the inequality $\norma{ \F{u_k} - \yd } >
\tau \delta$ holds, we use this inequality on the left hand side of the above
estimate and conclude that
$$
k^* < \left( \dfrac{\g}{\tau \delta} \right)^2 \dfrac{ \overline{\a} }
{8 \left( 1 - c \right) } + 1 \, .
$$
\textbf{Step 3:} In order to prove the convergence of the series in
\eqref{temp4}, notice that the right hand side of \eqref{temp4} does not
depend on $k^*$. \\
\textbf{Step 4:} Analogous to step 3 in the proof of theorem \ref{theo:nlic},
we use the Bregman distance, and remark \eqref{remark:update} to conclude that
\begin{eqnarray*}
\h{u_{k^*}} & \leq & \h{\ub} + \left| \inner{ \xi_0}{\ub - u_{k^*}}\right| \\
& & +  \sum_{i=0}^{k^*-2} \dfrac{1}{\a_i} \left| \hinner{\F{u_{i+1}} - \yd}
{\dF{u_{i+1}} \left( \ub - u_{k^*} \right) } \right| \nonumber \\
& & \ + \dfrac{1}{\a_{k^*-1}} \left| \hinner{\F{u_{k^*}} - \yd}{\dF{u_{k^* }}
\left( \ub - u_{k^*} \right) } \right| \, .
\end{eqnarray*}
In the sequel we estimate the three terms on the right hand side of this
inequality. For the first of them we have
$\lnorma{\ub - u_{k^*}} < \rho $. Indeed, on step $k^*-1 $ we have $u_{k^*}$
as minimizer of $\Jd{\a_k}{\cdot}$, thus $\Jd{\a_k}{u_{k^*}} \leq
\Jd{\a_k}{\ub}$. Rearranging the terms and discarding some positive terms,
it follows that $\db{\xi_{k^*-1}}{u_{k^*}}{u_{k^*-1}} <
\dfrac{\delta^2 }{2 \a_{k^*-1}} + \dfrac{\g^2}{8}$. Finally, we apply lemma
\ref{lemma:nli3} with $\delta < \overline{\delta} =
\sqrt{ 3/4 \g^2 \underline{\a} }$. \\
To estimate the last two terms we use Cauchy-Schwarz, assumption
\ref{assum:3}, lemma \ref{lemma:nli1} and \ref{lemma:nli2}, remark
\ref{remark:assumption} together with steps 1, 2 and 3 above. Summarizing,
we obtain
\begin{eqnarray*}
\h{u_{k^*}} & < & \h{\ub} + \rho \lnorma{\xi_0} +
\dfrac{\left( c + 3 + 4 \eta \rho  \right)} { \left( 1 - c \right)}
\dfrac{\g^2}{8} + \dfrac{ \tau \delta^2 \left( 1 + \eta \rho \right)
\left(  1 + \tau \right) }{ \underline{\a} } \, .
\end{eqnarray*}
\textbf{Step 5:} This step is very similar to step 4 in the proof of
theorem \ref{theo:nlic}. We just need to show that
$\F{u_{k^*_m}} \rightarrow y$. This convergence follows from the estimate
\begin{eqnarray*}
\norma{\F{u_{k^*_m}} - y} & \leq & \norma{\F{u_{k^*_m}} - y^{\delta_m} } +
\norma{y^{\delta_m} - y} \nonumber \\
& \leq & \left( 1 + \tau \right) \delta_m
\end{eqnarray*}
when $\delta_m$ goes to zero.
\end{proof}

\appendix
\begin{appendix}

\section{Definitions} \label{sec:app}

\begin{definition}
Given $\h{\cdot}$ a convex functional, one can define the Bregman distance with respect to $\hh$ between the elements $v, u \in \domf$ as
\begin{equation*}
\db{}{v}{u} =  \left\lbrace \db{\xi}{v}{u} \mid \xi \in \sg{u} \right\rbrace \ ,
\end{equation*}where $\sg{u}$ denotes the subdifferential of $\hh$ at $u$ and
\begin{equation*}
\db{\xi}{v}{u} = \h{v} - \h{u} - \inner{\xi}{v-u} \ .
\end{equation*} 
\end{definition}

We remark that $\inner{}{}$ denotes the standard dual pairing (duality product) with respect to $\U^* \times \U$.

Another important definition is the generalized solution, we introduce the notion of the \textit{$\hh$-minimizing solution} bellow.

\begin{definition}\label{def_sol_min}
An element $\ub \in \domf \cap \dom{\FF}$ is called an $\hh$-minimizing solution of \eqref{eq:princ} if it minimizes the functional $\hh$ among every possible solutions, that is,
\begin{equation*}
\ub = \argmin \left\lbrace \h{u} \tq \F{u} = y \right\rbrace \ .
\end{equation*}
\end{definition}

Whenever we need, we can choose the least-square solution instead the
standard solution $\F{u} = y$.

\end{appendix}

\section*{Acknowledgments}
The work of A.L. is supported by the Brazilian National Research
Council CNPq, grants 306020/2006--8, 474593/2007--0, and by the
Alexander von Humbolt Foundation AvH.


\bibliography{references}

\end{document}